\documentclass[letterpaper, 10 pt, conference]{ieeeconf}  
\IEEEoverridecommandlockouts                   
\usepackage{graphicx}       
\usepackage{amsfonts}
\DeclareMathSymbol{\shortminus}{\mathbin}{AMSa}{"39}
\let\labelindent\relax
\usepackage{enumitem}
\usepackage{multirow}
\usepackage{booktabs}
\usepackage{array}
\usepackage{float}
\usepackage{amsmath}
\usepackage{bm}
\usepackage{epsfig}
\usepackage{arydshln}
\usepackage{algpseudocode}
\usepackage{verbatim}
\usepackage{subfigure}
\usepackage{enumerate}
\usepackage{rotating}
\usepackage{color}
\usepackage{flushend} 
\usepackage[ruled,vlined,linesnumbered]{algorithm2e}
\usepackage{cite}
\usepackage{float}
\usepackage{soul}
\usepackage{fancyhdr}
\usepackage{url}
\usepackage{makecell} 
\usepackage{mathtools}
\usepackage{caption}
\usepackage{xcolor}

\usepackage[para,online,flushleft]{threeparttable}

\title{\LARGE \bf
Optimality of delayed reward attainment under uncertainty
}

\author{Chuwei Wang, Alexander Vladimirsky*, Anastasia Bizyaeva* % <-this % stops a space
\thanks{C.W. and A.B. are with the Sibley School of Mechanical and Aerospace Engineering, Cornell University, Ithaca, NY 14853, USA \texttt{\{cw996,anastasiab\}@cornell.edu}}% <-this % stops a space
\thanks{A.V. is with the Department of Mathematics, Cornell University, Ithaca, NY 14853, USA \texttt{vladimirsky@cornell.edu}}
\thanks{* A.V. and A.B. contributed equally to the manuscript.}
\thanks{This work was partially funded by the NSF DMS-2111522, and AFOSR FA9550-22-1-0528.} }%

\newtheorem{theorem}{Theorem}

\newtheorem{assumption}{Assumption}

\begin{document}
\maketitle
\thispagestyle{empty}
\pagestyle{empty}

\begin{abstract}
Decision making under uncertainty requires balancing 
the value of additional information against the cost of delaying the decision until that information becomes available.
This challenge is especially prevalent for embodied decisions such as 
those faced by
animals moving between food sources and robots selecting navigation targets, where movement through space shapes both the quality of information and the cost of acting.
We consider a version of this problem, where the decision is on choosing among rewards of uncertain value, and those rewards can only be attained through changing the decision maker's state sufficiently.
Our focus is on
interactions 
between reward learning and 
state evolution, restated  
as a deterministic optimal control problem with two
terminal alternatives, one with a known reward and the other learned gradually through noisy observations. 
We demonstrate that 
delayed reward attainment is often
optimal in a variety of systems,
especially so when the alternatives are initially difficult to distinguish and the cost of waiting is relatively low.  
For a subset of problems, we also derive sufficient conditions that allow shortening the planning horizon a priori, guaranteeing that further observations would not affect the value function.
Finally, when observation quality depends on the state, we show that it is often optimal for a controller to  evolve the state strategically toward more informative sensing locations before making a decision, in line with the behavior broadly observed in biological and engineered systems.
\end{abstract}
\section{Introduction}

Many biological and engineered systems must make irreversible decisions
under uncertainty. In such settings, an agent must balance the benefit
of making an early decision, thereby limiting additional costs of delaying, against the benefit of remaining uncommitted
while acquiring information that may improve the eventual choice.
This trade-off arises in foraging, navigation, target identification,
and resource allocation, where values of competing alternatives may
be only partially known. 
In this paper, we study 
different ways in which 
delaying a decision can be optimal when 
state evolution, information acquisition, and uncertain terminal rewards are all coupled. 

Normative theories of
decision making 
balance the value of
acting on current information against the expected value of future
information \cite{BOGACZ2007118}. 
Algorithmically, the optimal timing for choosing among immediately available uncertain rewards is also well-studied in
the literature on sequential sampling \cite{Chick12}.  
However, this 
cost/benefit analysis
becomes more subtle when selecting an alternative does not immediately yield its reward, but instead requires subsequent state evolution before the reward can be attained. 
For example, when animals or robots make decisions on alternatives with different physical target locations, 
rewards can only be obtained 
through changing the agent's physical position.
Actions such as movement maneuvers taken while the decision remains unresolved can then affect both the information available for making the decision and the cumulative cost of claiming the reward once it is selected.

Resulting trade-offs are well-known in both biological and engineered systems. Across species, longer decision times have been observed when competing alternatives are difficult to distinguish, including alternatives with similar subjective values in monkeys
\cite{setogawa2019neurons} and perceptual discrimination tasks in bumblebees and mice \cite{dyer2004bumblebees,sanders2012choice}. When decisions involve
movement, fruit flies, desert locusts, zebrafish, and mice can delay spatial choices while integrating information about competing alternatives\cite{pnasSridhar21,gorbonos2024geometrical,BENQUET20263523}. Similarly, in aircraft collision avoidance, uncertainty in 
an approaching aircraft's state can make it optimal to delay the 
advisory until the appropriate maneuver, such as climbing or descending, becomes clearer \cite{Kochenderfer2012NextGenerationAC, katz2025aircraft}.  
Interestingly, in some robotics applications, delayed choice is not just an emergent behavior, but might be deliberately engineered to preserve reachability to multiple candidate targets for as long as possible while maintaining upper bounds on the cost of attaining each of them \cite{Elangoddto22,11563843}.

Several prior papers have studied the emergence of delay as optimal behavior 
due to stochastic perturbations.
Optimal control formulations can lead to
delayed decisions because of the stochasticity in the dynamics of the controller itself \cite{PhysRevLett05,broek2012risk}, the dynamics of
obstacles/adversaries \cite{Kochenderfer2012NextGenerationAC, katz2025aircraft}, or the perceived dynamics of targets \cite{Haith2015HedgingYB}.  In contrast, we show that the delay can 
be optimal
even if the dynamics is fully deterministic but the 
terminal reward is uncertain.
Specifically, we consider two 
alternative decision targets,
with one yielding a known reward and the other an
uncertain reward inferred from noisy observations. 
As shown in Fig.~\ref{fig:intro}, delayed reward attainment emerges 
as optimal behavior
when the alternatives are initially difficult to distinguish 
and waiting is relatively inexpensive. 
This observation holds 
both for linear bounded-control dynamics and  nonlinear opinion dynamics with controlled attention rate,
suggesting that delayed commitment is a general
consequence of cost trade-offs rather than a feature of a particular 
model.

\begin{figure}[htbp]
    \centering
    \includegraphics[width=1\linewidth]{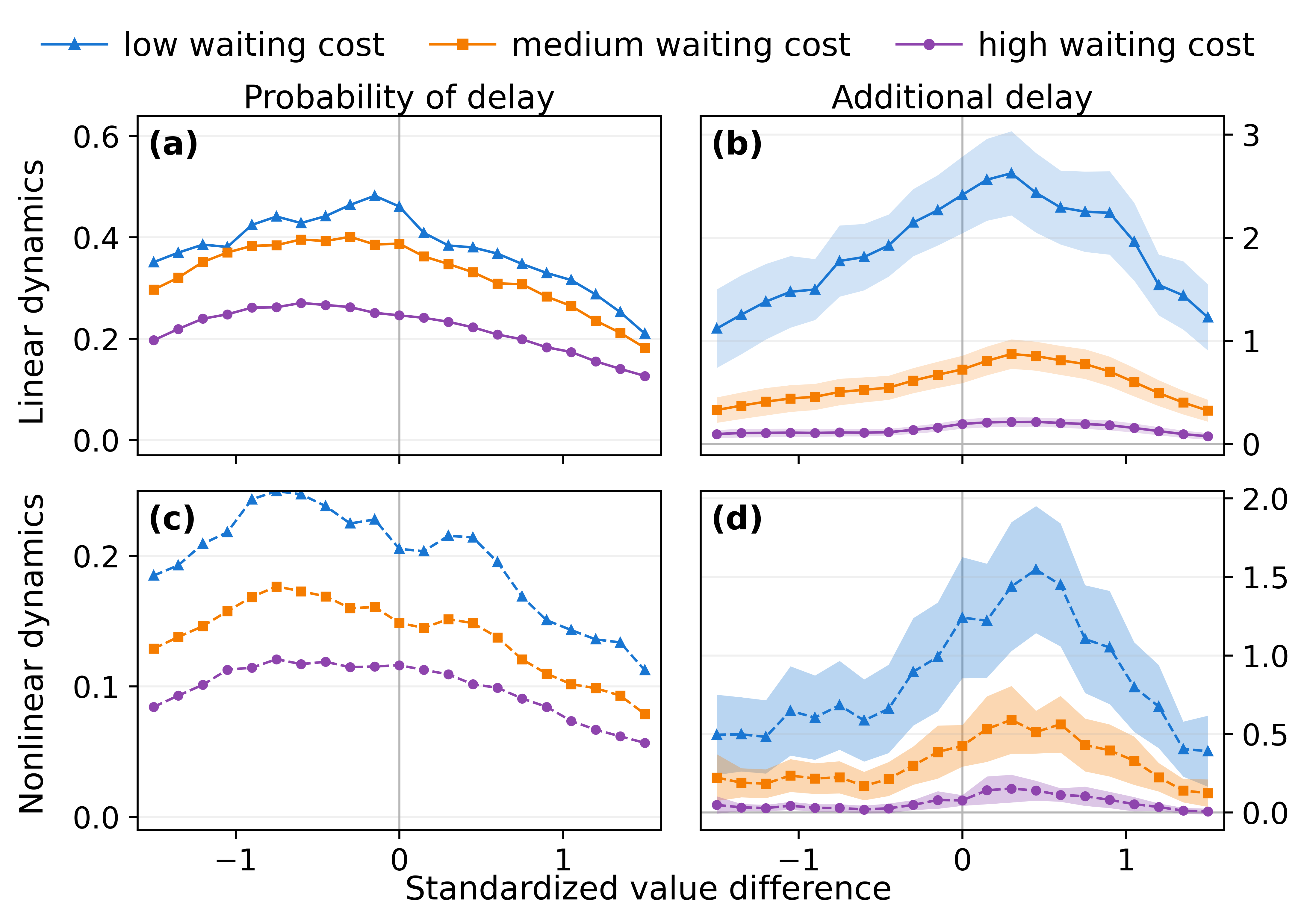}
    \caption{\small Delayed decision making emerges as an optimal strategy when the alternatives
    are difficult to distinguish 
    due to noisy observations
    and waiting is
    relatively
    inexpensive.  This behavior persists across both linear (top row) and nonlinear (bottom row)
    dynamics.
    The probability of delayed reward attainment
    (left) and the additional delay (right) are shown as functions of
    the standardized value difference, \((K_1-K_2)/\sigma\), with curves
    representing different waiting costs. Shaded regions indicate approximate 95\% confidence intervals for the mean additional delay time.
    Definitions, parameters, and numerical
    details are provided in \S~\ref{sec_IV_A}.}
    \label{fig:intro}
\end{figure}

Our optimal control approach based on dynamic programming
is described in \S\ref{s:problem}.
In \S\ref{s:horizon} we focus on a subset of problems, for which we can derive an a priori bound on observation horizon 
such that later observations can be disregarded without changing the optimal value.  
Numerical experiments in \S\ref{s:results} explore the range of parameter values for which delayed reward attainment is optimal. We also examine how {\em state-dependent observation quality} can induce
strategic state evolution that improves information acquisition before termination.  Our findings are summarized and future directions briefly described in \S\ref{s:conclusions}.

\section{Problem Formulation and numerics}
\label{s:problem}

\subsection{Dynamics and Observation Model}

We consider a controlled dynamical system
\begin{equation}
    \dot z=f(z,u;\mu),
    \qquad
    u\in\mathcal U,
    \label{eq:dynamics}
\end{equation}
where  \(z\in\mathcal Z=[z_{\min},z_{\max}]\) is the state,
\(u\) is the control input and  \(\mu\) represents the current estimate of an uncertain terminal reward, defined formally below.
The function \(f\) is
assumed to be continuous and locally Lipschitz in \(z\).

The decision maker can terminate the process by attaining one of two
terminal rewards available at the boundaries of the state domain. At the upper boundary \(z_{\max}\),
termination yields an uncertain reward \(K_1^\star\), whereas at the
lower boundary \(z_{\min}\), termination yields a known reward \(K_2\).
Thus, the corresponding terminal reward is
\begin{equation}
R_{\rm term}
=
\begin{cases}
K_1^\star, & z=z_{\max},\\
K_2, & z=z_{\min}.
\end{cases}
\label{eq:terminal_reward}
\end{equation}
Reaching a boundary makes the corresponding terminal reward available
but does not require termination; the controller may instead continue
inward when doing so is preferable. The process ends when the
controller terminates at either boundary.

Information about \(K_1^\star\) is acquired at a finite sequence of
sampling times
\(
0=t_1<t_2<\cdot\cdot\cdot<t_n.
\) If the process has not terminated by \(t_j\), the observation is
\begin{equation}
o_j
=
K_1^\star+\omega_j,
\qquad
\omega_j
\sim
\mathcal N
\!\left(
0,\sigma^2(z(t_j))
\right),
\label{eq:obs}
\end{equation}
where \(\sigma^2(z)\) denotes the observation-noise variance. The function
\(\sigma(z)\) includes both constant-variance and state-dependent
observation models considered later in the paper.

The belief is initialized using the first observation at $t_1=0$, with \(\mu_1 = o_1, P_1 = \sigma^2(z(t_1))\).  After each subsequent observation, the conditional distribution of \(K_1^\star\) remains Gaussian. Let
\((\mu_j,P_j)\) denote its posterior mean and variance after the
\(j\)th observation. For \(j = 2,
\ldots,n\), the recursive Bayesian updates are
\begin{align}
P_j
&=
\left(
P_{j-1}^{-1}
+
\sigma^{-2}(z(t_j))
\right)^{-1},
\label{eq:bayes_P}
\\
\mu_j
&=
\mu_{j-1}
+
\frac{P_{j-1}}
{P_{j-1}+\sigma^2(z(t_j))}
\left(
o_j-\mu_{j-1}
\right).
\label{eq:bayes_mu}
\end{align}

The posterior distribution summarizes the information obtained from
the observation history through its mean $\mu$ and variance $P$. We refer to $(\mu,P)$ as the belief state. Together with the physical state $z$, these variables define the
augmented state $x=(z,\mu,P)$. Between consecutive observations,
$(\mu,P)$ remains unchanged while $z$ evolves according to
\eqref{eq:dynamics}.

\subsection{Optimal Control Problem}

Let \(T_s\) denote the termination time, at which the corresponding terminal reward is attained. 
The objective is to maximize the expected terminal reward while
penalizing decision delay and control effort. The return associated with
an admissible control policy is
\begin{equation}
J
=
R_{\rm term}
-
\int_t^{T_s}
(d+cu^2(s))\,ds,
\label{eq:return}
\end{equation}
where \(d>0\) penalizes delay and \(c>0\) penalizes control effort.
The corresponding value function is
\begin{equation}
V(z,\mu,P,t)\!
=\!
\sup_{u(\cdot)}
\mathbb E
\!\left[
J\!
\mid\!
z(t)\!=z,\,\!
\mu(t)\!=\mu,\,\!
P(t)\!=P
\right].
\label{eq:value}
\end{equation}

At the boundaries of the admissible state domain, the controller may terminate and attain the
corresponding terminal reward. Conditional on the current posterior, the expected reward available upon termination are $K_2$ at $z_{\min}$ and $\mu$ at $z_{\max}$,
since \(\mathbb E[K_1^\star\mid\mu,P]=\mu\).
Because reaching or
starting at a boundary does not require termination, the controller
may instead continue inward. Accordingly, the boundary conditions are interpreted in the
viscosity sense,
\begin{equation}
V(z_{\min},\mu,P,t)\ge K_2,
\qquad
V(z_{\max},\mu,P,t)\ge \mu.
\label{eq:boundary_condition}
\end{equation}

Between observation times, the belief $(\mu,P)$ is fixed, and the value function
satisfies the Hamilton--Jacobi--Bellman equation
\begin{equation}
\!\frac{\partial V}{\partial t}\!
+\!
\max_{u\in\mathcal U}\!
\left\{\!
-d\!-\!cu^2\!
+\!
\frac{\partial V}{\partial z}\!f(z,u;\mu)\!
\right\}\!\!
=\!0,\ z\in\mathcal Z
\label{eq:HJB}
\end{equation}

After the final observation at $t_n$, no further belief updates occur and the
remaining control problem is time invariant. Since $P$ affects the
value only through future belief updates, the stationary
post-observation value depends only on $(z,\mu)$ and is denoted by
$V_{\rm tail}(z,\mu)$. For each fixed posterior mean $\mu$, it
satisfies
\begin{equation}
\!\!\max_{u\in\mathcal U}\!
\left\{\!
-\!d\!-\!cu^2\!
+\!
\frac{\partial V_{\rm tail}}{\partial z}
f(\!z,\!u;\!\mu\!)\!
\right\}\!
=\!0,
\ z\in\mathcal Z,\!
\label{eq:stationary_tail}
\end{equation}
with boundary conditions
\begin{equation}
    V_{\rm tail}(z_{\min},\mu)\ge K_2,
    \qquad
    V_{\rm tail}(z_{\max},\mu)\ge\mu,
    \label{eq:tail_boundary_condition}
\end{equation}
again interpreted in the viscosity sense. We use $t_j^-$ and $t_j^+$ to denote the instants immediately before and after the observation at $t_j$, respectively.
The stationary solution then provides the
terminal condition at \(t_n^+\) for backward recursion,
\begin{equation}
V(z,\mu_n,P_n,t_n^+)
=
V_{\rm tail}(z,\mu_n).
\label{eq:tail_terminal_condition}
\end{equation}

The value function is then constructed backward through the preceding observation intervals. At each observation time $t_j$,
$j=2,\ldots,n$, the physical state remains continuous while the belief
changes from $(\mu_{j-1},P_{j-1})$ at $t_j^-$ to $(\mu_j,P_j)$ at
$t_j^+$. Accordingly, the value function satisfies the jump condition
\begin{equation}
V\!(z,\!\mu_{j-1},\!P_{j-1},\!t_j^-)\!
=\!
\mathbb E\!\left[
V\!(\!z,\mu_j,\!P_j,\!t_j^+)\!\!
\mid\!
\mu_{j-1},\!P_{j-1}
\right].
\label{eq:jump}
\end{equation}

\subsection{Linear and Nonlinear Decision Dynamics}
To examine whether delayed decision-making depends on the particular form of the decision dynamics, we consider two models under the same reward structure, observation process, and running-cost formulation.

The first model is the linear integrator
\begin{equation}
    \dot z = u,
    \label{eq:linear_dynamics}
\end{equation} with the admissible state domain $ \mathcal Z_{\rm lin}=[z_{\min},z_{\max}] =[-\varepsilon,\varepsilon].$ At \(z=z_{\min}\) and \(z=z_{\max}\), the known and uncertain terminal
rewards, respectively, become available.

The second model is nonlinear opinion dynamics (NOD) model \cite{Bizyaeva23},
\begin{equation}
    \dot z\!
    =\!
    -z \!+\! h(u)\!
    + \tanh\!\left(
        u\left[\alpha(z-h(u))\!+\!b\right]
    \right),
    \ 
    h(u)=\beta u,
    \label{eq:nonlinear_dynamics}
\end{equation}
where $\alpha$ is a constant and
\begin{equation}
    b=\mu-K_2
    \label{eq:belief_bias}
\end{equation}
represents the current belief bias toward the uncertain alternative. The input $u$ is interpreted as attention and is always positive.  The NOD model captures a nonlinear symmetry-breaking mechanism
associated with a pitchfork bifurcation, while the bias \(b\) favors
one alternative over the other. The parameter \(\beta>0\) determines
the magnitude of the control-dependent shift in the nonlinear
dynamics and is chosen to be small.

For the nonlinear model, we consider the admissible state domain
\begin{equation}
    \mathcal Z_{\rm NOD}
    =
    [z_{\min},z_{\max}]
    =
    [-\varepsilon+\beta u_{\max},\,\varepsilon],
    \label{eq:nonlinear_domain}
\end{equation}
where $u_{\max}=\max_{u\in\mathcal U}|u|$. The asymmetry of this
domain accounts for the control-dependent shift $\beta u$ in the
nonlinear dynamics. The lower boundary is therefore shifted by the
maximum possible displacement $\beta u_{\max}$, while the underlying
decision thresholds remain symmetric. As in the linear model, the
known reward $K_2$ is available at $z=z_{\min}$ and the uncertain
reward $K_1^\star$ is available at $z=z_{\max}$. Reaching either
boundary does not require termination; the controller may instead
continue inward.

\subsection{Numerical Implementations}
The optimal policy is computed by backward dynamic programming on
discretized state, belief, and control spaces. We first compute the stationary tail value \(V_{\rm tail}\) from \eqref{eq:stationary_tail}, with the boundary condition \eqref{eq:tail_boundary_condition}.  At each boundary, termination is
compared with admissible inward continuation. The resulting
stationary solution provides the terminal condition at \(t_n^+\) for
the backward recursion.

The value function is then propagated backward between consecutive
observation times using a semi-Lagrangian discretization of
\eqref{eq:HJB} \cite{Falcone13}. If the transition would leave the admissible state domain,
the trajectory is terminated at the corresponding boundary, and the
running cost is accumulated only until the boundary is reached.

At each observation time, the expectation in \eqref{eq:jump} is
evaluated as a Gaussian convolution along the posterior-mean
coordinate. Conditioned on the pre-observation belief, the posterior
mean satisfies
\begin{equation}
    \mu_j
    \mid
    \mu_{j-1},P_{j-1}
    \sim
    \mathcal N\!\left(
        \mu_{j-1},P_{j-1}-P_j
    \right),
    \label{eq:posterior_mean_transition}
\end{equation}
where $P_j$ is given by \eqref{eq:bayes_P}. Hence, the jump condition
is evaluated as
\begin{equation}
\begin{aligned}
&V(z,\mu_{j-1},P_{j-1},t_j^-)
\\[-1mm]
&\quad =
\int_{\mathbb R}
V(z,\mu',P_j,t_j^+)
\,\mathcal N\!\left(
\mu';\mu_{j-1},P_{j-1}-P_j
\right)d\mu'.
\end{aligned}
\label{eq:jump_convolution}
\end{equation}
The integral is evaluated numerically on the discretized belief space for each state \(z\), after which the value function is propagated backward to the preceding observation time.

\section{Truncating the observation horizon}
\label{s:horizon}

The formulation in \S\ref{s:problem} considers a prescribed sequence of observation times $t_1,\ldots,t_n$. In general, whether further delay is justified must be determined dynamically based on the information
accumulated up to the current observation.
We now derive sufficient conditions under which one can determine a priori that observations
beyond a certain point cannot justify their additional waiting cost. Such observations can then be removed without changing the optimal value, allowing an optimal policy to be obtained from the truncated
problem.

For this analysis, we assume that observations 
have constant observation variance \(\sigma^2\)
and are acquired every $\tau>0$ seconds,
$t_j=(j-1)\tau.$
We further assume that the physical dynamics are independent of the belief 
\begin{equation}
    \dot z=f(z,u).
    \label{eq:mu_independent_dynamics}
\end{equation}
Under the Bayesian update of the reward estimate in \S\ref{s:problem}, after the  $j$th observation, the posterior satisfies
\begin{equation}
    K_1^\star\mid(\mu_j,P_j)
    \sim\mathcal N(\mu_j,P_j),
    \qquad
    P_j=\frac{\sigma^2}{j}.
    \label{eq:posterior_j}
\end{equation}
Consequently,
\begin{equation}
    \mu_{j+1}\mid\mu_j
    \sim
    \mathcal N\!\left(
        \mu_j,\frac{\sigma^2}{j(j+1)}
    \right).
    \label{eq:belief_transition}
\end{equation}

The benefit of an additional observation decreases as information
accumulates. To quantify this benefit, define
\begin{equation}
    \Delta_j(\mu_j)
    =
    \mathbb E\!\left[
        \max(\mu_{j+1},K_2)
        \mid\mu_j
    \right]
    -
    \max(\mu_j,K_2).
    \label{eq:delta_j}
\end{equation}
Thus, $\Delta_j(\mu_j)$ is the expected improvement in the larger
terminal reward resulting from the $(j+1)$st observation, relative to
selecting between the two alternatives using the current belief.
Using \eqref{eq:belief_transition}, its maximum over the current
posterior mean is
\begin{equation}
    \Delta_j^{\max}
    :=
    \max_{\mu_j}\Delta_j(\mu_j)
    =
    \frac{\sigma}{\sqrt{2\pi j(j+1)}},
    \label{eq:delta_max}
\end{equation}
where the maximum is attained at $\mu_j=K_2$. Hence
$\Delta_j^{\max}$ decreases monotonically with $j$.

The quantity $\Delta_j$ measures the information-induced improvement
in terminal reward. To relate this 
to the control problem,
consider a relaxation in which, immediately after the $(j+1)$st
observation, the controller may receive the larger posterior expected
terminal reward without incurring any additional cost to reach the
corresponding boundary. Thus, $\Delta_j$ provides an upper bound on
the increase in expected overall value
attributable to one additional observation.

We now compare this information benefit with the value attainable without using future observations. 
Under such \emph{immediate-commitment policy}, controller selects the alternative with the larger current expected terminal reward $\max(\mu,K_2)$ and then moves to it
while ignoring further observations even if they occur before the corresponding boundary  is reached.
Deciding to obtain another observation, in contrast, requires remaining uncommitted until $t_{j+1}$ and therefore incurs at least the waiting cost $d\tau$ 
regardless of additional costs of changing the state.

Two restrictions allow these quantities to be compared uniformly. First, the belief-independent dynamics in \eqref{eq:mu_independent_dynamics} allow the benefit of an additional observation to be bounded through its effect on the reward estimate, independently of the state evolution. Second, we require a uniform bound on the cost of attaining the alternative selected under
immediate commitment.

\begin{assumption}[Bounded reward-attainment cost]
\label{assump:bounded_commit}
For every admissible nonterminal state $z$ and belief $\mu$, there
exists an immediate-commitment policy that reaches the boundary
associated with $\max(\mu,K_2)$ and incurs accumulated running cost no
greater than some $\bar C<\infty$ before termination.
\end{assumption}

By Assumption~\ref{assump:bounded_commit}, there exists a feasible
policy in the tail problem that attains expected terminal reward $\max(\mu,K_2)$
with running cost at most $\bar C$. Since
$V_{\rm tail}$ is the optimal value over all policies that use no
further observations, it cannot be smaller than the value achieved by
this policy.  Meanwhile, because no further information is acquired in the tail problem, the expected terminal reward is at most $\max(\mu,K_2)$ and the nonnegative running cost can only decrease the value.
Therefore,
\begin{equation}
    \max(\mu_j,K_2)-\bar C
    \le
    V_{\rm tail}(z,\mu_j)
    \le
    \max(\mu_j,K_2).
    \label{eq:tail_bounds}
\end{equation}

We derive the truncation condition by comparing the values of waiting for another observation and immediate commitment. Recall that the relaxation allows the controller, after the next observation, to receive the larger posterior expected terminal reward without incurring any additional reward-attainment cost.  Thus, the value of waiting until $t_{j+1}$ is bounded above by
\begin{equation}
    \mathbb E\!\left[
        \max(\mu_{j+1},K_2)\mid\mu_j
    \right]
    -d\tau,
\end{equation}
whereas Assumption~\ref{assump:bounded_commit} guarantees an immediate-commitment policy with the value of at least $ \max(\mu_j,K_2)-\bar C$.
Therefore, waiting for the next observation cannot improve the value if
\begin{equation}
    \mathbb E\!\left[
        \max(\mu_{j+1},K_2)\mid\mu_j
    \right]
    -d\tau
    \le
    \max(\mu_j,K_2)-\bar C.
\end{equation}
Using \eqref{eq:delta_j}, this condition is equivalent to $ \Delta_j(\mu_j) \le d\tau-\bar C$.
Since $\Delta_j(\mu_j)\le\Delta_j^{\max}$, a sufficient condition
uniformly over the current posterior mean is
\begin{equation}
    \Delta_j^{\max}
    \le
    d\tau-\bar C.
    \label{eq:truncation_condition}
\end{equation}
Because $\Delta_j^{\max}$ decreases monotonically with $j$, once
\eqref{eq:truncation_condition} holds, it also holds for every
subsequent observation.

To compare problems with different truncation points, let
\begin{equation}
    V^{(k)}(z,\mu,P,t),
    \qquad k=1,\ldots,n,
    \label{eq:horizon_value}
\end{equation}
denote the value function when $t_k$ is the final available observation time. Thus $V^{(n)}$ is the value function of the original problem with all $n$ prescribed  observations, whereas $V^{(k)}$, for $k<n$, is the value function of the problem truncated
after the $k$th observation.

\begin{theorem}[Finite observation horizon]
\label{thm1}
Suppose Assumption~\ref{assump:bounded_commit} holds and
$d\tau>\bar C$. 
Define
\begin{equation}
    N
=
\left\lceil
\frac{1}{2}\left(-1+\sqrt{
1+
\frac{2\sigma^2}{\pi(d\tau-\overline{C})^2}
}\right)
\right\rceil
\label{eq:nstar}
\end{equation}

If $N<n$, then, for every $j=N,\ldots,n$,
\begin{equation}
    V^{(n)}(z,\mu_j,P_j,t_j^+)
    =
    V_{\rm tail}(z,\mu_j)
    \label{eq:backward_claim}
\end{equation}
for every admissible $z$ and posterior mean $\mu_j$, with $P_j=\sigma^2/j$. In particular, the
observations $t_{N+1},\ldots,t_n$ can be disregarded without changing the
optimal value at $t_N^+$:
\begin{equation}
    V^{(n)}(z,\mu_N,P_N,t_N^+)
    =
    V^{(N)}(z,\mu_N,P_N,t_N^+).
    \label{eq:truncation_result}
\end{equation}
\end{theorem}

\begin{proof}
At the final observation time $t_n^+$, no future observations remain, so
by \eqref{eq:tail_terminal_condition},
\begin{equation}
    V^{(n)}(z,\mu_n,P_n,t_n^+)
    =
    V_{\rm tail}(z,\mu_n),
    \label{eq:base_case}
\end{equation}
so \eqref{eq:backward_claim} holds for $j=n$.

Starting from the established base case at $t_n^+$, we propagate the
equality backward one observation interval at a time. Specifically,
if
\begin{equation}
    V^{(n)}(z,\mu_{j+1},P_{j+1},t_{j+1}^+)
    =
    V_{\rm tail}(z,\mu_{j+1})
    \label{eq:induction_hypothesis}
\end{equation}
holds at $t_{j+1}^+$ for every admissible $z$ and posterior mean
$\mu_{j+1}$, then we show below that the same equality must hold at
$t_j^+$.
 
Starting from $(z,\mu_j,P_j)$ at $t_j^+$, the controller may proceed without using further observations to make subsequent decisions or remain active until $t_{j+1}$ and obtain the next observation. The optimal value in the absence of further observations is
$V_{\rm tail}(z,\mu_j)$. This does not require termination before $t_{j+1}$; the controller may continue to evolve the state, but its subsequent decisions do not depend on additional observations.

Let $W_j(z,\mu_j,P_j)$ denote the optimal value among policies that use the observation at $t_{j+1}$. Any such
policy must remain active until $t_{j+1}$ and therefore incurs the
delay cost $d\tau$. Neglecting the nonnegative control-effort cost over
$[t_j,t_{j+1}]$ and applying induction gives
\begin{align}
    W_j(z,\mu_j,P_j)
    &\le
    -d\tau+
    \mathbb E[
        V_{\rm tail}(z_{j+1},\mu_{j+1})
        \mid z,\mu_j, P_j]
    \nonumber\\
    &\le
    -d\tau+
    \mathbb E[
        \max(\mu_{j+1},K_2)\mid\mu_j, P_j]
    \nonumber\\
    &\le
    \max(\mu_j,K_2)
    +\Delta_j^{\max}-d\tau,
    \label{eq:wait_bound}
\end{align}
where the second inequality follows from \eqref{eq:tail_bounds}, and the last follows from  \eqref{eq:delta_j}--\eqref{eq:delta_max}.

By \eqref{eq:nstar} and the monotonic decrease of
$\Delta_j^{\max}$,
\begin{equation}
    \Delta_j^{\max}
    \le
    \Delta_{N}^{\max}
    \le
    d\tau-\bar C,
    \qquad j\ge N.
    \label{eq:cutoff_condition}
\end{equation}
Thus, from \eqref{eq:wait_bound} and \eqref{eq:tail_bounds},
\begin{equation}
    W_j(z,\mu_j,P_j)
    \le
    \max(\mu_j,K_2)-\bar C
    \le
    V_{\rm tail}(z,\mu_j).
    \label{eq:wait_dominated}
\end{equation}
Therefore, for $j\ge N$, using the next observation cannot yield a
higher value than proceeding without further observations. Since
$V_{\rm tail}(z,\mu_j)$ is the optimal value when no further
observations are used, the optimal value of the original problem is
\begin{equation}
    V^{(n)}(z,\mu_j,P_j,t_j^+)
    =
    V_{\rm tail}(z,\mu_j).
    \label{eq:backward_step}
\end{equation}
Applying this argument successively from $j=n-1$ down to $j=N$
proves \eqref{eq:backward_claim}.

Finally, setting $j=N$ gives
\begin{equation}
    V^{(n)}(z,\mu_N,P_N,t_N^+)
    =
    V_{\rm tail}(z,\mu_N).
\end{equation}
In the problem truncated after $t_N$, the value at $t_N^+$ is also
given by the tail value, by \eqref{eq:tail_terminal_condition}
\begin{equation}
    V^{(N)}(z,\mu_N,P_N,t_N^+)
    =
    V_{\rm tail}(z,\mu_N).
\end{equation}
Hence,
\begin{equation}
    V^{(n)}(z,\mu_N,P_N,t_N^+)
    =
    V^{(N)}(z,\mu_N,P_N,t_N^+),
\end{equation}
which proves \eqref{eq:truncation_result}.
\end{proof}

Under these additional assumptions, Theorem~\ref{thm1} rules out indefinite information-driven indecision. Although uncertainty can make it optimal to delay reward attainment,
the benefits of further information eventually become too small to justify the cost of waiting.  This provides a finite horizon beyond which future observations need not to be considered when computing the optimal policy.

As a simple example, consider dynamics \eqref{eq:linear_dynamics} with $c = 0.08$, $d = 0.05$, $\varepsilon = 0.1$, $u_{\max}  = 1$, $\tau = 1.0$, and $\sigma = 0.5$. For these parameters, 
a simple upper bound on the
reward-attainment cost is
 \(\bar C=4\varepsilon\sqrt{cd}=0.0253<d\tau=0.05,\) so the conditions of Theorem~\ref{thm1} are satisfied. The resulting 
 cutoff is $N = 8$. Thus, all observations after the eighth can be disregarded without changing the optimal value.

\section{Numerical Experiments}
\label{s:results}
The code used to generate the numerical results in this section, along
with additional examples, is available at \url{https://github.com/ChuweiW/OC_NOD}
\subsection{Delayed Commitment Across Different Dynamics}
\label{sec_IV_A}

We first examine whether delayed reward attainment persists across distinct
decision dynamics. We consider the linear \eqref{eq:linear_dynamics} and nonlinear  \eqref{eq:nonlinear_dynamics} models under a
constant-variance observation model, $\sigma^2(z)=\sigma^2$, while allowing
the remaining model and numerical parameters to be selected separately
for each system. For both models, we consider a finite sequence of three observations. Figure~\ref{fig:trajectory_dynamics} shows representative
optimal trajectories for both models. At each of the three observation times, the
posterior variance $P_j=\frac{\sigma^2}{j}$ is common to all branches, while the labeled
posterior means $\mu_j$ represent different realizations of the belief
update caused by different observation values.

\begin{figure}[htbp]
    \centering
    \includegraphics[width=1\linewidth]{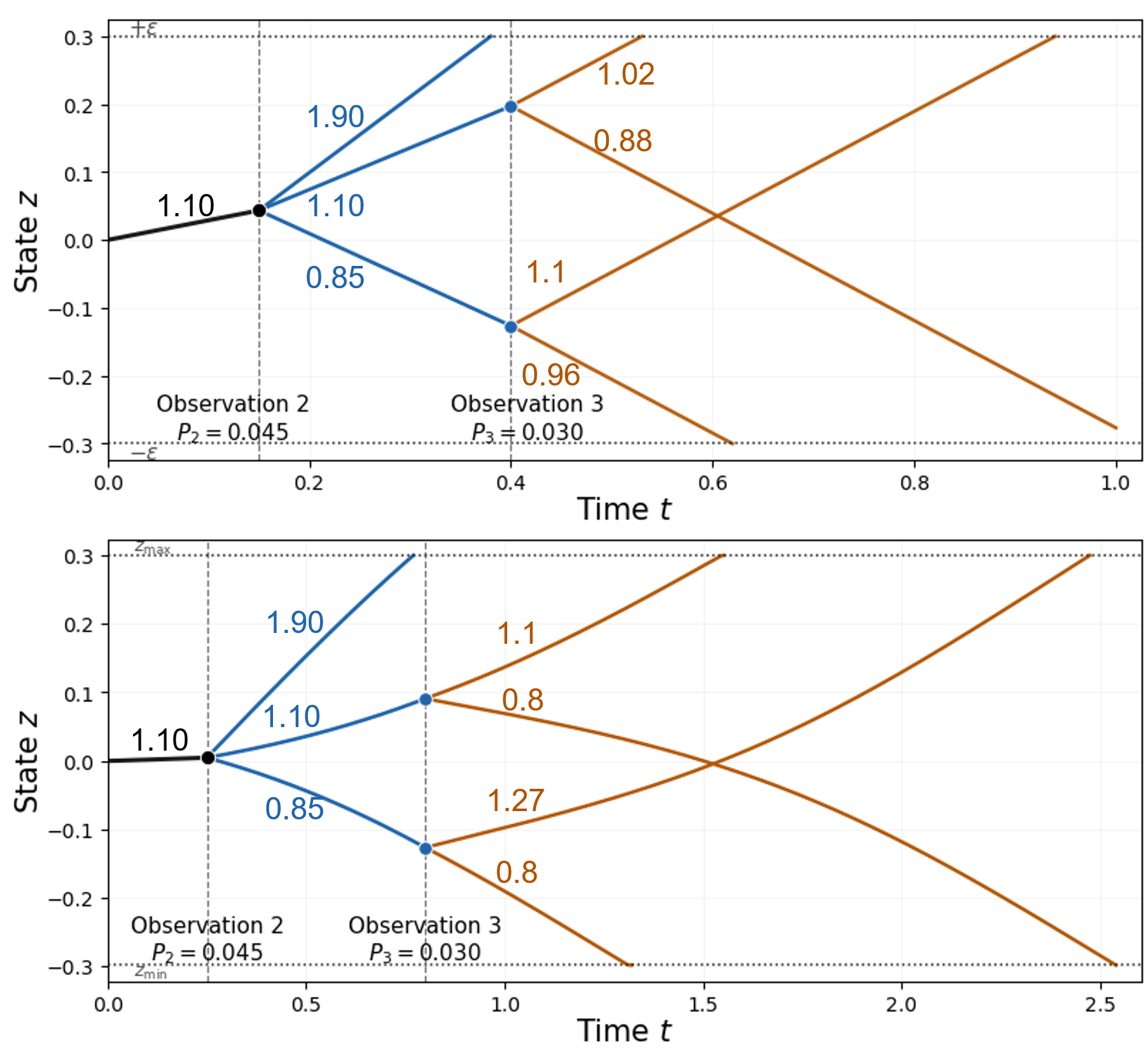}
    \caption{\small Representative optimal trajectories under the linear (top) and
    nonlinear (bottom) dynamics, illustrating how observation-induced posterior updates affect the control response. Black curves show the shared trajectory after the first observation; blue and brown trajectories
    show branches following the second and third observation. Numerical
    labels indicate the corresponding posterior means. Dashed vertical lines
    mark observation times $t_1=0$, $t_2=0.15$, and $t_3=0.4$ for linear model and $t_1=0$, $t_2=0.25$, and $t_3=0.8$ for the nonlinear model. Dotted horizontal lines denote the boundaries of state domain. Both panels use $z_0=0$, $K_2=1$, $d=0.05$,
    and $\sigma=0.3$.  For the linear model, $c=0.08$, $u\in[-2,2]$ and $z\in[-0.3,0.3]$. For the nonlinear model, $c=0.06$, $\alpha=1.5$, $\beta=0.001$,
    $u\in[0,2]$ and $z\in[-0.298,0.3]$.}
    \label{fig:trajectory_dynamics}
\end{figure}

Following the second observation, the trajectories illustrate how the
optimal response depends on both the posterior mean and the confidence
in that estimate. When the posterior mean strongly favors the upper
alternative, as for $\mu_2=1.90$ in both dynamics, the controller commits to the upper
alternative and terminates before the third observation. When the
posterior favors the lower alternative, as for $\mu_2=0.85$, the
controller instead begins moving toward the lower boundary while
postponing decisions until further information is received. The center
branch provides a particularly clear illustration of the role of
confidence: the posterior mean remains unchanged at $\mu=1.10$ after
the second observation, while the posterior variance decreases. Despite
the unchanged mean belief, the optimal control changes as confidence
in that belief increases. Thus, the policy depends on the full belief
state $(\mu,P)$, rather than on the posterior mean alone. At the
same time, the controller remains uncommitted despite having sufficient
time to attain that reward before the next observation, showing that this delay is an
optimal response to uncertainty rather than a consequence of limited
reachability.
Following the third observation, the problem reduces to the stationary
tail problem, and the controller proceeds toward the terminal
favored by the updated belief. The branching nevertheless illustrates how delayed reward attainment can arise from remaining uncommitted, which preserves the ability to redirect the motion when new information arrives. 

Although the trajectories differ in shape under the linear and nonlinear dynamics, both exhibit the same qualitative behavior:
uncertainty can make it optimal to delay selecting an alternative, whereas sufficiently informative observations can lead to earlier decision
and reward attainment. We next quantify this behavior in a simpler setting with at most two observations. As a baseline, we compute a
one-observation policy, under which the controller receives only the
initial observation and subsequently acts with a fixed belief. We compare
this baseline with a two-observation policy, which anticipates an
additional future observation and can therefore delay commitment to make
use of the information it provides. Each policy is optimized separately
under its corresponding observation schedule.

For the results in Fig.~\ref{fig:intro}, we set
$K_2=1$, $\sigma=0.25$, $c=0.1$, $u_{\max}=2$, and
$\varepsilon=0.3$, and consider waiting costs
$d\in\{0.001,0.01,0.05\}$. For the linear model, we use
$\mathcal{U}=[-u_{\max},u_{\max}]$ and
$\mathcal{Z}_{\rm lin}=[-\varepsilon,\varepsilon]$. For the nonlinear
model, we use $\mathcal{U}=[0,u_{\max}]$,
$\mathcal{Z}_{\rm NOD}=[-\varepsilon+\beta u_{\max},\varepsilon]$,
$\alpha=1.5$, and $\beta=0.001$. Thus, with the parameters above,
the state domains are $[-0.3,0.3]$ and $[-0.298,0.3]$ for the linear
and nonlinear models, respectively. For each waiting cost and dynamical
model, we first determine $t_{\mathrm{ref}}$ from the corresponding
one-observation baseline and then set the second observation time to
$t_2=t_{\mathrm{ref}}$, as described below.

The first measure reported in Fig.~\ref{fig:intro} is the probability of delayed reward attainment. For each waiting cost and dynamical model, we define a reference time $t_{\mathrm{ref}}$ from the reward-attainment times of the corresponding one-observation policy. Reward-attainment time is the time at which a trajectory reaches the boundary associated with the attained reward.
The reference times are computed by pooling across sampled initial states, reward differences, and initial
observations. For the linear
model, \(t_{\mathrm{ref}}\) is the largest sampled reward-attainment time.
For the nonlinear model, where reward-attainment times exhibit greater variability, \(t_{\mathrm{ref}}\) is the earliest time by
which at least \(90\%\) of the sampled one-observation trajectories have
attained a reward. The resulting reference times are listed in
Table~\ref{tab:d_tobs}. At each reward difference, the probability of
delayed reward attainment is estimated as the fraction of trajectories under the two-observation policy that have not attained a reward by \(t_{\mathrm{ref}}\).  Thus, this measure quantifies how often the
availability of an additional observation extends reward attainment beyond the characteristic time scale of the one-observation baseline.

The second measure quantifies the additional delay in reward attainment under the two-observation policy.
Trajectories under the one- and two-observation policies are paired using the same initial state $z_0$ and initial observation, and
we define the additional delay as
\begin{equation}
    \Delta T = T_{\mathrm{two}}-T_{\mathrm{one}},
    \label{eq:additional_delay}
\end{equation}
where \(T_{\mathrm{one}}\) and \(T_{\mathrm{two}}\) denote the corresponding
reward-attainment times. Figures~\ref{fig:intro}(b) and
\ref{fig:intro}(d) report the mean $\Delta T$, with the shaded regions
indicating  $95\%$ confidence intervals for the mean. Thus, $\Delta T>0$ indicates that the availability of future information delays reward attainment relative to the one-observation baseline.

To express the reward differences relative to observation uncertainty, the horizontal axis in Fig.~\ref{fig:intro} shows the standardized value difference
\[
    \frac{K_1^\star-K_2}{\sigma},
    \label{eq:standardized_value_difference}
\]
where \(K_1^\star\) is the true value of the uncertain reward,
\(K_2=1\) is the known reward, and \(\sigma=0.25\) is the observation
standard deviation.
Values near zero correspond to alternatives with similar rewards
relative to the observation noise, whereas increasingly positive or
negative values correspond to larger reward differences in favor of the
uncertain or known alternative, respectively.

\begin{table}[t]
\centering
\caption{\small Second observation times $t_2 = t_{\mathrm{ref}}$ for each waiting cost $d$ and dynamics.}
\label{tab:d_tobs}
\begin{tabular}{c c c}
\hline
Waiting cost $d$ & Linear $t_{\mathrm{ref}}$ & Nonlinear $t_{\mathrm{ref}}$ \\
\hline
0.001 (low waiting cost) & 6.005 & 6.670 \\
0.010 (medium waiting cost) & 2.005 & 4.150 \\
0.050 (high waiting cost)& 0.865 & 2.475 \\
\hline
\end{tabular}
\end{table}

\subsection{Effect of State-Dependent Observation Quality}
\begin{figure*}[ht]
    \centering
    \includegraphics[width=\textwidth]{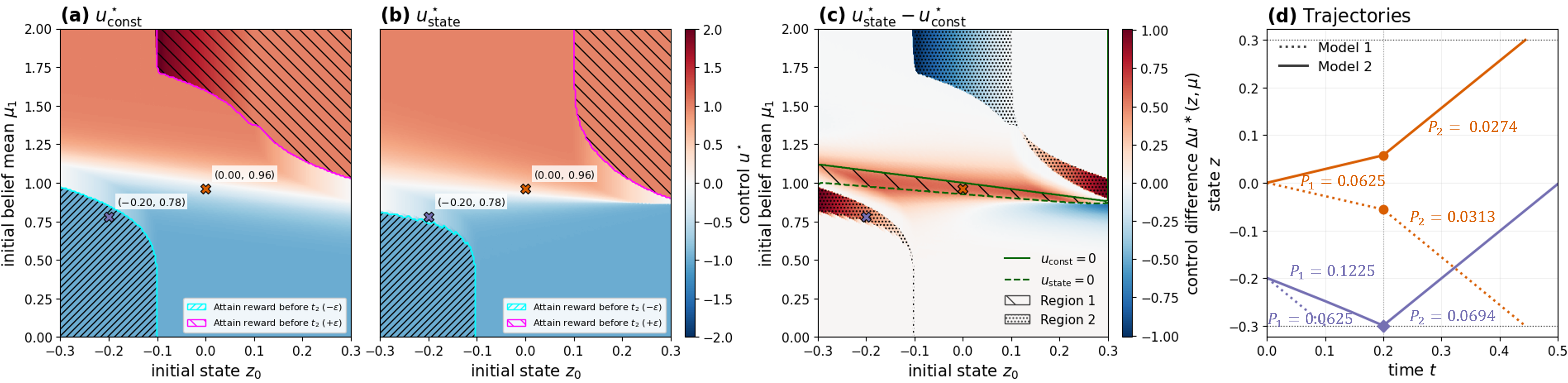}
    \caption{\small Effect of state-dependent observation quality on the
    optimal policy. (a) Initial optimal control $u^\star_{\rm const}$ for Model~1 (constant
    observation variance). (b) Initial optimal control $u^\star_{\rm state}$ for Model~2 (state-dependent observation variance). Cyan and magenta striped regions indicate initial conditions for which the system terminates at $-\varepsilon$ and $+\varepsilon$, respectively, and attains the corresponding terminal reward before $t_2=0.2$, regardless of the second observation.
    (c) Difference
    $u^\star_{\rm state}-u^\star_{\rm const}$; solid and dashed green
    curves denote $u^\star_{\rm const}=0$ and
    $u^\star_{\rm state}=0$, respectively. Regions~1 and~2 identify
    changes in the initial control direction and in reward attainment
    before $t_2$, respectively. (d) Representative trajectories for
    the initial conditions marked by crosses in (a)--(c).
    Dotted and solid curves correspond to 
    constant-variance (Model~1) and  state-dependent (Model~2) sensing, respectively; same-colored trajectories
    use identical observation realizations. For the orange trajectories,
    the second observation is $o_2=1.06$, yielding
    $\mu_2=1.010$ and $1.016$ under Model~1 and~2, respectively. For the purple Model~2 trajectory,
    $o_2=1.80$ yields $\mu_2=1.222$.
    Parameters: $c=0.1$, $d=0.1$, $\varepsilon=0.3$, $u_{\max}=2$,
    $\sigma_0=0.25$, $\sigma_{\min}=0.1$, and
    $\sigma_{\max}=0.4$.}
    \label{fig:obsmodel_compare}
\end{figure*}

Thus far, our experiments have assumed constant observation quality with information gained by combining multiple observations over time.  In many 
spatial decision problems, however, sensing quality
depends on the decision maker's location.  For example, an animal approaching a potential food source or a person moving toward an object of interest may obtain more reliable information as the distance to the object decreases. This motivates a state-dependent observation model in which motion toward an uncertain alternative can improve the quality of subsequent observations. 
As we show below, this dependence can induce strategic optimal motion toward  an option not because it is currently preferred, but because approaching it provides information that can help determine whether it should be selected.

To isolate this effect, 
we compare the optimal behavior in two systems with identical linear dynamics \eqref{eq:linear_dynamics}, reward structures, cost parameters, and observation times. The only difference is in the observation variance. We consider only two observations at $t_1 = 0$ and $t_2 = 0.2$. In \emph{Model~1}, observations have constant standard deviation,
\begin{equation}
    \sigma(z)\equiv\sigma_0.
\end{equation}
In \emph{Model~2}, the observation standard deviation depends on the state according to
\begin{equation}
    \sigma(z)
    =
    \sigma_{\min}
    +
    (\sigma_{\max}-\sigma_{\min})
    \left(
        \frac{z_{\max}-z}{z_{\max}-z_{\min}}
    \right),
    \qquad
    z\in\mathcal Z_{\rm lin},
    \label{eq:state_variance_result}
\end{equation}
where \(0<\sigma_{\min}\leq\sigma_{\max}\) and
$\sigma(\frac{z_{\max}+z_{\min}}{2}) = \sigma_0$.  Thus the two models have the same observation quality at the midpoint, whereas in Model~2 the observation standard deviation decreases linearly as the state moves toward $z_{\max}$ , where the uncertain reward $K_1^\star$ is available.

Figures~\ref{fig:obsmodel_compare}(a) and
\ref{fig:obsmodel_compare}(b) show the initial optimal control $ u^*(z,\mu)$ under
Model~1 and Model~2, respectively.  The cyan and magenta hatched regions contain initial conditions for which the system terminates and attains the reward at $z_{\min}$ and
$z_{\max}$, respectively, before the second observation at $t_2=0.2$. Comparing the two panels shows that state-dependent observation quality
can affect both the direction and magnitude of the initial control and, consequently, whether reward is attained before the next observation.
These effects are illustrated in Fig.~\ref{fig:obsmodel_compare}(c).

To compare the two policies directly, we define
\begin{equation}
    \Delta u^*(z,\mu)
    =
    u^*_{\mathrm{state}}(z,\mu)
    -
    u^*_{\mathrm{const}}(z,\mu).
    \label{eq:policy_difference}
\end{equation}

Figure~\ref{fig:obsmodel_compare} (c) shows
\(\Delta u^*(z,\mu)\) at the initial time together with the zero-control curves for the two models.
Striped ``Region~1'' identifies initial conditions starting from which the two policies move in opposite directions. 
Dotted ``Region~2'' identifies a different effect: the two policies
initially move in the same direction, but at different rates, with the slower one representing a deliberate information gathering delay.
In this region, one policy attains a terminal reward before $t_2$, while the other remains in the interior and acquires the second observation.

The representative trajectories in Fig.~\ref{fig:obsmodel_compare}
(d) illustrate these two effects. Each
same-colored pair starts from the same initial condition and receives
the exact same observation realizations.
Nevertheless, the policies differ because
the posterior variances  differ  and depend on the state $z$
in  Model~2.
In particular, moving toward $z_{\max}$ reduces $\sigma(z)$ and therefore provides a
lower-variance, more informative measurement of $K_1^\star$.

For the representative initial condition in Region~1, shown by the
orange trajectories in Fig.~\ref{fig:obsmodel_compare}(d),
the initial measurement yields $\mu_1 < K_2=1$, so the current belief favors the known reward $K_2$. In Model~1, the controller accordingly begins moving toward $z_{\min}$, as shown by the dotted orange trajectory. 
Because the reward differential is small relative to $\sigma_0,$ however, the controller moves gradually while retaining the ability to respond to the second observation.  In Model~2 (trajectory shown in solid orange), the controller instead moves toward  $z_{\max}$, despite the current belief favoring $K_2$, because doing so improves the quality of the next observation of $K_1^\star$.  
The potential information gain is sufficient to justify the motion away from the currently favored alternative, even though doing so can increase the subsequent cost of attaining $K_2$ if the next observation continues to favor it. 
Interestingly, in this example the second observation does reveal $K_1^\star$ as a slightly better option.  Under Model~1, however, this small reward advantage is insufficient to offset the additional cost of reversing direction from the state already reached, and the controller continues toward $z_{\min}$.

For the representative initial condition in Region~2, shown by the
purple trajectories in Fig.~\ref{fig:obsmodel_compare}(d), both policies initially move toward $z_{\min}$, but the optimal behavior of Model~1 is far more aggressive.
At the initial state $z_0 = -0.2$, $\sigma(z_0)> \sigma_0$, so Model~2 has higher initial uncertainty than Model~1 and hedges its bets, approaching $z_{\min}$ just in time for the second observation and ready to change course if $K_1^\star$ turns out to be significantly better.

\section{Conclusion}
\label{s:conclusions}
Choosing between uncertain options 
is hard even for a fully rational decision maker.  
Motivated by embodied decision making in biology and robotics, we focused on a ``simple'' choice between
two rewards, one of which has an uncertain value that has to be estimated through a sequence of noisy observations.  
Unlike in the standard sequential sampling
setting (e.g., \cite{Chick12}), our rewards cannot be claimed instantaneously
and have to be attained by sufficiently changing the decision maker's internal state, which can also directly affect the quality of observations.  Balancing the costs of controlled state evolution, information acquisition, and the risk of selecting a suboptimal reward results in a nuanced optimization problem exhibiting several types of strategic behavior.  Using dynamic programming, we demonstrated that it is often {\em optimal} to delay the reward attainment, particularly when the respective values of two rewards are initially hard to distinguish and the cost of waiting for additional observations is relatively low.  We showed that this is generally true for very different types of controlled dynamics and observation models.  
For a subset of problems in this class, we have also derived 
an a priori upper bound
on a number of useful observations.
When observation quality is state-dependent, we showed that
it can be optimal to approach an uncertain reward purely for information-gathering reasons, even if its value is currently viewed as somewhat inferior.  
Our results provide a scientific insight on why delays in decision making under uncertainty may have evolved as ubiquitous features of natural behavior.  They also suggest an important design principle for autonomous decision making across domains: when information is uncertain and stakes are high, it pays 
to hedge your bets.  
Future work will extend the presented framework to multiple uncertain rewards, higher-dimensional state spaces, 
stochastic controlled dynamics, and multiple (cooperating or competing) decision makers.

\bibliographystyle{ieeetr}
\bibliography{references}
\end{document}